\documentclass[a4paper,11pt]{article}

\usepackage{amsmath}
\usepackage{amssymb}
\usepackage{amsthm}
\usepackage{array}
\usepackage{amscd}
\usepackage{cases}
\usepackage{bm}
\usepackage[all]{xy}
\usepackage{authblk}
\usepackage{cancel}
\usepackage{ulem}
\usepackage{mathrsfs}

\theoremstyle{plain}
\newtheorem{thm}{Theorem}[section]
\newtheorem{prop}[thm]{Proposition}

\newtheorem{conj}[thm]{Conjecture}
\theoremstyle{definition}

\theoremstyle{definition}

\theoremstyle{remark}

\numberwithin{equation}{section}

\title{A $q$-analogue of the matrix sixth Painlev\'e system}
\author{Hiroshi KAWAKAMI\thanks{\texttt{kawakami@gem.aoyama.ac.jp}}}
\affil{College of Science and Engineering, Aoyama gakuin university,
5-10-1 Fuchinobe, Chuo-ku, Sagamihara-shi, Kanagawa 252-5258, Japan.}
\date{}

\begin{document}
\maketitle

\begin{abstract}
We derive a $q$-analogue of the matrix sixth Painlev\'e system via a connection-preserving deformation of
a certain Fuchsian linear $q$-difference system.
In specifying the linear $q$-difference system,
we utilize the correspondence between linear differential systems and linear $q$-difference systems
from the viewpoint of the spectral type.
The system of non-linear $q$-difference equations thus obtained can also be regarded as
a non-abelian analogue of Jimbo-Sakai's $q$-$P_{\mathrm{VI}}$.
\end{abstract}

\paragraph{Mathematics Subject Classifications (2010).}34M55, 34M56, 33E17
\paragraph{Key words.}$q$-difference equation, connection-preserving deformation, isomonodromic deformation,
Painlev\'e-type equation, integrable system.


\section{Introduction}\label{sec:intro}

The Painlev\'e equations are non-linear second order ordinary differential equations that define new special functions.
Originally, the Painlev\'e equations were classified into six equations.
We denote them by $P_{\mathrm{I}},\ldots,P_{\mathrm{VI}}$.
Among the Painlev\'e equations, the sixth Painlev\'e equation $P_{\mathrm{VI}}$ is the ``source'' in the sense that
all the other Painlev\'e equations can be obtained from $P_{\mathrm{VI}}$ through degenerations.

Since the 1990s, there have been many generalizations of the Painlev\'e equations in the literature
such as discretizations, higher dimensional analogues, quantizations, and so on.
As for second order (or two-dimensional) discrete Painlev\'e equations,
there is an algebro-geometric theory~\cite{Sak2001} which provides a unified account of them.
According to the theory, they fall into three types of equations, namely,
additive difference, multiplicative difference ($q$-difference),
and elliptic difference equations.
The Painlev\'e (differential) equations are understood through the continuous limit of those discrete Painlev\'e equations.
In this sense we can say that discrete Painlev\'e equations are more fundamental than the Painlev\'e differential equations.

Recently, Painlev\'e-type differential equations with four-dimensional phase space are classified from the viewpoint of isomonodromic deformations of linear differential equations~\cite{HKNS, K3, K4, K5}.
This study shows that,
in the four-dimensional case, there are four ``sources'' as extensions of the sixth Painlev\'e equation.
Namely, they are
\begin{itemize}
\item the Garnier system~\cite{G}, which is a classically known multivariate extension of $P_{\mathrm{VI}}$,

\item the Fuji-Suzuki-Tsuda system~\cite{FS1, Ts2014}, which is an extension of $P_{\mathrm{VI}}$ with the affine Weyl group symmetry of type $A$,

\item the Sasano system~\cite{Ss}, which is an extension of $P_{\mathrm{VI}}$ with the affine Weyl group symmetry of type $D$,

\item the matrix sixth Painlev\'e system~\cite{B2, HKNS}, which is a non-abelian extension of $P_{\mathrm{VI}}$.
\end{itemize}
Note that each of the four equations has its extensions defined in arbitrary even dimensions.
These four families are expected to play an important role in the theories of integrable systems, special functions and so on.

Our next purpose is,
motivated by the two-dimensional case,
to develop a unified framework of discrete Painlev\'e-type equations in higher dimensions.
As a first step, we investigate a correspondence of
Painlev\'e-type differential equations and discrete Painlev\'e-type equations.

Among the above four families, $q$-analogues of the Garnier system,
the Fuji-Suzuki-Tsuda system, and the Sasano system have already been constructed and studied by several authors~\cite{Sak2005, Sz, Ts2010, Ma}.
The aim of this paper is to obtain a $q$-analogue of the matrix sixth Painlev\'e system (abbreviated to matrix $P_{\mathrm{VI}}$).

The matrix $P_{\mathrm{VI}}$ is derived from an isomonodromic deformation of a certain Fuchsian differential equation.
The linear equation is specified in terms of the spectral type.
The spectral type of a Fuchsian differential equation is data on the multiplicities of the characteristic exponents of a Fuchsian differential equation~\cite{Os} (see Section~\ref{sec:mpVI}).
On the other hand,
the notion of spectral type is also defined for Fuchsian linear $q$-difference equations~\cite{SY}.
Thus we have the spectral type for both Fuchsian differential and $q$-difference equations.

The author observed the relationship between the two spectral type (see Conjecture~\ref{thm:correspondence_of_ST}),
and this correspondence is expected to be applicable to a construction of a $q$-analogue of a given isomonodromic deformation equation.
That is, suppose we are given an isomonodromic deformation equation of a Fuchsian differential equation.
Specify a linear $q$-difference equation which corresponds to the Fuchsian differential equation in terms of the spectral type.
Then consider the connection-preserving deformation of the linear $q$-difference equation.
The connection-preserving deformation~\cite{JS} is a discrete counterpart of the isomonodromic deformation of linear differential equations.
This paper provides an example of the construction.

This paper is organized as follows.
In Section~\ref{sec:mpVI}, we review the matrix sixth Painlev\'e system.
In Section~\ref{sec:LqDE}, we explain the connection-preserving deformation of a linear $q$-difference system and
the spectral type for Fuchsian $q$-difference systems.
We also present a conjecture concerning a correspondence between the spectral type of $q$-difference systems and that of differential systems.
In Section~\ref{sec:derivation}, we derive a $q$-analogue of the matrix $P_{\mathrm{VI}}$.
In Section~\ref{sec:CL}, we show that the resulting system of non-linear $q$-difference equations
actually has a continuous limit to the matrix $P_{\mathrm{VI}}$.

\section{The matrix sixth Painlev\'e system}\label{sec:mpVI}

In this section we review the matrix $P_{\mathrm{VI}}$~\cite{K2, K5}.
The matrix $P_{\mathrm{VI}}$ is derived from the isomonodromic deformation of a certain Fuchsian differential system.
First we give a brief account of the isomonodromic deformation of Fuchsian systems.

Consider a system of Fuchsian differential equations:
\begin{equation}\label{eq:linear_system}
\frac{dY}{dx}=
\mathcal{A}(x, u)Y, \quad
\mathcal{A}(x, u)=\sum_{j=1}^n \frac{\mathcal{A}_j}{x-u_j} \quad (\mathcal{A}_j \in M_m(\mathbb{C})).
\end{equation}
We put $\mathcal{A}_\infty:=-\sum_{j=1}^n \mathcal{A}_j$.
For simplicity, we assume that $\mathcal{A}_j$'s $(j=1, \ldots, n, \infty)$ are semisimple
and that $\mathcal{A}_\infty$ is diagonal.
Positions of the singular points $u=(u_1, \dots, u_n)$ are regarded as deformation parameters.
It is known that the normalized local solution of \eqref{eq:linear_system} at $x=\infty$ is an isomonodromic solution
(that is, associated monodromy matrices are independent of $u$) if and only if
the following system is completely integrable:
\begin{align}
\left\{
\begin{aligned}
\frac{\partial Y}{\partial x}&=\mathcal{A}(x,u)Y, \\
\frac{\partial Y}{\partial u_i}&=\mathcal{B}_i(x,u)Y, \quad \mathcal{B}_i(x,u)=-\frac{\mathcal{A}_i}{x-u_i} \quad (i=1,\ldots,n).
\end{aligned}
\right.
\end{align}
Then the compatibility condition of the above
\begin{equation}
\left\{
\begin{aligned}
&\frac{\partial \mathcal{A}(x,u)}{\partial u_i}-\frac{\partial \mathcal{B}_i(x,u)}{\partial x}+[\mathcal{A}(x,u), \mathcal{B}_i(x,u)]=O \\
&\frac{\partial \mathcal{B}_i(x,u)}{\partial u_j}-\frac{\partial \mathcal{B}_j(x,u)}{\partial u_i}+[\mathcal{B}_i(x,u), \mathcal{B}_j(x,u)]=O
\end{aligned}
\right.
\end{equation}
gives a system of non-linear differential equations satisfied by the entries of $\mathcal{A}_i$'s.
Here we put $[A, B]:=AB-BA$.
We can write the compatibility condition more explicitly as follows:
\begin{equation}\label{eq:Sch}
\left\{
\begin{aligned}
\frac{\partial A_i}{\partial u_j}&=
\frac{[A_j, A_i]}{u_j-u_i} \quad (j \neq i),\\
\frac{\partial A_i}{\partial u_i}&=
-\sum_{j \neq i}\frac{[A_j, A_i]}{u_j-u_i}.
\end{aligned}
\right.
\end{equation}
This system is the isomonodromic deformation equation of \eqref{eq:linear_system},
which is called the Schlesinger system.

Fuchsian systems are classified by their \textit{spectral types}.
Let the eigenvalues of $\mathcal{A}_j$ be $\theta^j_k \ (k=1, \ldots, \ell_j)$, and
let their multiplicities be $m^j_k \ (k=1, \ldots, \ell_j)$ respectively.
Regarding $\mathcal{P}_j=m^j_1, \ldots, m^j_{\ell_j}$ as a partition of $m$,
we have a $(n+1)$-tuple of partitions $\mathcal{P}_1 \, ; \, \ldots \, ; \, \mathcal{P}_n \, ; \, \mathcal{P}_\infty$.
We call this $(n+1)$-tuple of partitions of $m$ the spectral type of \eqref{eq:linear_system}.

The matrix $P_{\mathrm{VI}}$ is derived from the isomonodromic deformation of the following Fuchsian system:
\begin{equation}\label{eq:Fuchs_mp}
\frac{dY}{dx}=
\mathcal{A}(x, t)Y, \quad
\mathcal{A}(x, t)=
\frac{\mathcal{A}_0}{x}+\frac{\mathcal{A}_1}{x-1}+\frac{\mathcal{A}_t}{x-t},
\end{equation}
where $\mathcal{A}_0$, $\mathcal{A}_1$, and $\mathcal{A}_t$ are $2m \times 2m$ matrices satisfying the following conditions
\begin{equation}\label{eq:eigen_condition}
\mathcal{A}_0 \sim
\begin{pmatrix}
O_m & O_m \\
O_m & \theta^0 I_m
\end{pmatrix},\quad
\mathcal{A}_1 \sim
\begin{pmatrix}
O_m & O_m \\
O_m & \theta^1 I_m
\end{pmatrix},\quad
\mathcal{A}_t \sim
\begin{pmatrix}
O_m & O_m \\
O_m & \theta^t I_m
\end{pmatrix},
\end{equation}
and
\begin{equation}\label{residue_infty}
\mathcal{A}_\infty=-(\mathcal{A}_0+\mathcal{A}_1+\mathcal{A}_t)=
\mathrm{diag}(\overbrace{\theta^\infty_1, \ldots ,\theta^\infty_1}^{m}, \overbrace{\theta^\infty_2, \ldots, \theta^\infty_2}^{m-1}, \theta^\infty_3).
\end{equation}
Thus the spectral type of the Fuchsian system \eqref{eq:Fuchs_mp} is $m,m \, ; \, m,m \, ; \, m,m \, ; \, m,m-1,1$
and the number of accessory parameters is $2m$.
Taking the trace of \eqref{residue_infty}, we have the Fuchs relation
\begin{equation}
m(\theta^0+\theta^1+\theta^t+\theta^\infty_1)
+(m-1)\theta^\infty_2+\theta^\infty_3=0.
\end{equation}
The linear systems satisfying the conditions \eqref{eq:eigen_condition} and \eqref{residue_infty} can be parametrized as follows~(\cite{K2}):
\begin{equation}\label{eq:LDE_mPVI_res}
\begin{split}
\mathcal{A}_{\xi}&=
(U \oplus I_m)^{-1}X^{-1}\hat{\mathcal{A}}_{\xi}X(U \oplus I_m)
\quad(\xi=0,1,t),\\
\hat{\mathcal{A}}_0&=
\begin{pmatrix}
I_m \\
O_m
\end{pmatrix}
\begin{pmatrix}
\theta^0 I_m & \frac{1}{t}Q-I_m
\end{pmatrix},\quad
\hat{\mathcal{A}}_1=
\begin{pmatrix}
I_m \\
PQ-\Theta
\end{pmatrix}
\begin{pmatrix}
\theta^1 I_m-PQ+\Theta & I_m
\end{pmatrix},\\
\hat{\mathcal{A}}_t&=
\begin{pmatrix}
I_m \\
tP
\end{pmatrix}
\begin{pmatrix}
\theta^t I_m+QP & -\frac{1}{t}Q
\end{pmatrix},
\end{split}
\end{equation}
where the matrix $\Theta$ is given by
\begin{align}
\Theta=
\begin{pmatrix}
\theta^\infty_2 I_{m-1}& O\\
O & \theta^\infty_3
\end{pmatrix}.
\end{align}
The matrices $Q$ and $P$ satisfy the following commutation relation
\begin{equation}\label{eq:canonical_var}
[P,Q]=(\theta+\theta^\infty_1)I_m+\Theta.
\end{equation}
Here we put $\theta=\theta^0+\theta^1+\theta^t$.
Hereafter we often write a scalar matrix $kI$ as $k$ if there is no danger of confusion.
The matrix $U \in \mathrm{GL}_m(\mathbb{C})$ is a gauge variable,
and $X$ is given by
$X=
\begin{pmatrix}
I_m & O \\
Z & I_m
\end{pmatrix}$
where
\begin{align}
Z&=(\theta^{\infty}_1-\Theta)^{-1}
[-\theta^1(QP+\theta+\theta^{\infty}_1)
+(QP+\theta+\theta^{\infty}_1)^2-t(PQ+\theta^t)P].
\end{align}

As mentioned in the above,
the isomonodromic deformation of \eqref{eq:Fuchs_mp} is equivalent to the compatibility condition of the following Lax pair:
\begin{equation}\label{eq:Lax_mp6}
\left\{
\begin{aligned}
\frac{\partial Y}{\partial x}&=
\mathcal{A}(x, t)Y, \\
\frac{\partial Y}{\partial t}&=\mathcal{B}(x, t)Y,
\end{aligned}
\right.
\end{equation}
where $\mathcal{B}(x, t)=-\frac{\mathcal{A}_t}{x-t}$.
Then the compatibility condition
\begin{equation}
\frac{\partial \mathcal{A}(x,t)}{\partial t}-\frac{\partial \mathcal{B}(x,t)}{\partial x}+[\mathcal{A}(x,t), \mathcal{B}(x,t)]=O
\end{equation}
of the Lax pair \eqref{eq:Lax_mp6} can be explicitly written as
{\small
\begin{align}
t(t-1)\frac{dQ}{dt}&=(Q-t)PQ(Q-1)+Q(Q-1)P(Q-t) \nonumber \\
&\quad+(\theta^0+1)Q(Q-1)+(\theta+2\theta^\infty_1-1)Q(Q-t)+\theta^t(Q-1)(Q-t), \label{eq:MP6-q}\\
t(t-1)\frac{dP}{dt}&=-(Q-1)P(Q-t)P-P(Q-t)PQ-PQ(Q-1)P \nonumber \\
&\quad-\left[(\theta^0+1)\{P(Q-1)+QP\}+(\theta+2\theta^\infty_1-1)\{P(Q-t)+QP\}+\theta^t\{P(Q-t)+(Q-1)P\}\right] \nonumber \\
&\quad-(\theta+\theta^\infty_1)(\theta^0+\theta^t+\theta^\infty_1), \label{eq:MP6-p}\\
t(t-1)\frac{dU}{dt}&=\left\{(Q-t)(PQ+QP)+(2\theta^0+\theta^1+2\theta^t+2\theta^\infty_1)Q-\theta^t t\right\}U.
\end{align}
}
Moreover, equations \eqref{eq:MP6-q} and \eqref{eq:MP6-p} can be written in the following form:
\begin{align}\label{eq:matPVI_Hamform}
\frac{dq_{ij}}{dt}&=\frac{\partial H^{\mathrm{Mat},m}_{\mathrm{VI}}}{\partial p_{ji}},\quad
\frac{dp_{ij}}{dt}=-\frac{\partial H^{\mathrm{Mat},m}_{\mathrm{VI}}}{\partial q_{ji}}\quad (i, j=1,\ldots,m)
\end{align}
in terms of the entries of $Q=(q_{ij})$ and $P=(p_{ij})$.
The Hamiltonian is given by
\begin{multline}
t(t-1)H^{\mathrm{Mat},m}_{\mathrm{VI}}
\\=
\mathrm{tr}\Big[Q(Q-1)(Q-t)P^2+
\{(\theta^0+1 -(\theta+\theta^\infty_1+\Theta))Q(Q-1)+\theta^t(Q-1)(Q-t)
\\
+(\theta+2\theta^\infty_1-1)Q(Q-t)\}P
+(\theta+\theta^\infty_1)(\theta^0+\theta^t+\theta^\infty_1)Q\Big].
\end{multline}
We call the equations \eqref{eq:MP6-q} and \eqref{eq:MP6-p}, or equivalently the Hamiltonian system \eqref{eq:matPVI_Hamform},
\textit{the matrix sixth Painlev\'e system} (matrix $P_{\mathrm{VI}}$).
We construct a $q$-analogue of the equations \eqref{eq:MP6-q} and \eqref{eq:MP6-p} in Section~\ref{sec:derivation}.

\section{Linear $q$-difference systems and their deformations}\label{sec:LqDE}
In this section,
we collect some facts about linear $q$-difference equations that will be used later.
\subsection{Connection-preserving deformation of linear $q$-difference systems}
Let $q$ be a complex number satisfying $0 < |q| < 1$.
We consider a system of linear $q$-difference equations
\begin{equation}\label{eq:LqDE}
Y(qx)=A(x)Y(x),
\end{equation}
where $A(x)=A_N x^N+\cdots+A_1x+A_0$ is an $m \times m$ matrix with polynomial entries.
We assume that the matrices $A_N$ and $A_0$ are semisimple and invertible.
Using a gauge transformation $Y(x) \to GY(x)$ with a constant matrix $G$,
we can assume that $A_N$ is diagonal without loss of generality.
We set
\begin{equation}
A_0=G_0\,\mathrm{diag}(\theta_1, \ldots, \theta_m)\,G_0^{-1}, \quad
A_N=\mathrm{diag}(\kappa_1, \ldots, \kappa_m)
\end{equation}
and
\begin{equation}
D_0=\frac{1}{\log q}\mathrm{diag}( \log \theta_1, \ldots, \log \theta_m ), \quad
D_\infty=\frac{1}{\log q}\mathrm{diag}( \log \kappa_1, \ldots, \log \kappa_m).
\end{equation}
We also assume the non-resonant condition, that is, for any $i, j$
\begin{equation}
\theta_j / \theta_i, \ \kappa_j / \kappa_i \notin q^{\mathbb{Z}_{\ge 1}}=\{ q^n \mid n \in \mathbb{Z}_{\ge 1} \}.
\end{equation}
Under the above assumptions,
it is known that the system \eqref{eq:LqDE} has the following solutions at $x=0$ and $x=\infty$:
\begin{prop}[\cite{Car}]\label{thm:local_sol}
The system \eqref{eq:LqDE} has the following solutions
\begin{align}
Y_0(x)&=G_0\hat{Y}_0(x)x^{D_0}, \\
Y_\infty(x)&=q^{\frac{N}{2}u(u-1)}\hat{Y}_\infty(x)x^{D_\infty} \quad (u=\log x/\log q),
\end{align}
where $\hat{Y}_0(x)$ (resp. $\hat{Y}_\infty(x)$) is a invertible matrix
whose entries are holomorphic at $x=0$ (resp. $x=\infty$) and
satisfies $\hat{Y}_0(0)=I_m$ (resp. $\hat{Y}_\infty(\infty)=I_m$).
\end{prop}
By using (\ref{eq:LqDE}), we have for any $k \in \mathbb{Z}_{\ge 1}$
\begin{align}
\hat{Y}_0(x)&=G_0^{-1}A(x)^{-1}A(qx)^{-1}\cdots A(q^{k-1}x)^{-1}G_0\hat{Y}_0(q^k x)G_0^{-1}A_0^kG_0, \\
\hat{Y}_\infty(x)^{-1}&=q^{-\frac{k(k+1)N}{2}}x^{kN}A_N^k\hat{Y}_\infty(q^{-k}x)^{-1}A(q^{-k}x)^{-1}\cdots A(q^{-2}x)^{-1}A(q^{-1}x)^{-1}.
\end{align}
These expressions show that ${\hat{Y}_0(x)}^{-1}$ and $\hat{Y}_\infty(x)$ can be analytically continued
to $\mathbb{P}^1 \setminus \{ 0, \infty \}$,
while $\hat{Y}_0(x)$ and ${\hat{Y}_\infty(x)}^{-1}$ can be \textit{meromorphically} continued to
the same domain.
Let $\alpha_j \ (j=1, \ldots, mN)$ be the zeros of $\det A(x)$.
Then the matrix $\hat{Y}_0(x)$ may have poles at
\begin{equation}
\{ q^{-k} \alpha_j \,|\, k \in \mathbb{Z}_{\ge 0}, \ j=1, \ldots, mN \},
\end{equation}
and the matrix ${\hat{Y}_\infty(x)}^{-1}$ may have poles at
\begin{equation}
\{ q^k \alpha_j \,|\, k \in \mathbb{Z}_{\ge 1}, \ j=1, \ldots, mN \}.
\end{equation}

The connection matrix $P(x)$ is defined on $\mathbb{P}^1 \setminus \{ 0, \infty \}$ as follows:
\begin{equation}
Y_\infty(x)=Y_0(x)P(x).
\end{equation}
Obviously, $P(x)$ is pseudo-constant, that is, $P(qx)=P(x)$.

Next we consider a deformation of the system.
We can choose some of eigenvalues $\theta_i$'s, $\kappa_j$'s, and zeros $\alpha_j$'s of $\det A(x)$ as deformation parameters.
For simplicity, we assume that the number of deformation parameters is one.
Let us denote such a parameter by $t$.
Then, in general, the connection matrix $P(x)$ depends on $t$; $P=P(x, t)$.

The \textit{connection-preserving deformation} requires that the connection matrix $P(x,t)$ be pseudo-constant with respect to
the parameter $t$.
The following holds:
\begin{prop}\label{thm:pseudo-const}
The connection matrix $P(x,t)$ is pseudo-constant with respect to $t$, namely $P(x,qt)=P(x,t)$, if and only if
\begin{align}
Y_\infty(x,qt)Y_\infty(x,t)^{-1}=Y_0(x,qt)Y_0(x,t)^{-1}.
\end{align}
\end{prop}
\begin{proof}
This is immediate from
\begin{equation}
Y_\infty(x,qt)Y_\infty(x,t)^{-1}=Y_0(x,qt)P(x,qt)P(x,t)^{-1}Y_0(x,t)^{-1}.
\end{equation}
\end{proof}
We put $B(x,t):=Y_\infty(x,qt)Y_\infty(x,t)^{-1}=Y_0(x,qt)Y_0(x,t)^{-1}$.
Then the solutions $Y(x,t)=Y_0(x,t), \,Y_\infty(x,t)$ satisfy
\begin{equation}\label{eq:deform}
Y(x, qt)=B(x, t)Y(x, t).
\end{equation}
Under some general conditions, the matrix $B(x,t)$ is shown to be rational in $x$.
Conversely, if the solutions $Y(x,t)=Y_0(x,t), \,Y_\infty(x,t)$ satisfy \eqref{eq:deform}
for some matrix $B(x, t)$ which is rational in $x$,
then the corresponding connection matrix is pseudo-constant in $t$.

Thus (under some conditions) the connection-preserving deformation is equivalent to the existence of $B(x,t)$ which is rational in $x$ such that the
following pair
\begin{equation}\label{eq:qLax}
\left\{
\begin{aligned}
Y(qx,t)=A(x,t)Y(x,t) \\
Y(x,qt)=B(x,t)Y(x,t)
\end{aligned}
\right.
\end{equation}
is compatible.
Then, analogous to the isomonodromic deformation of linear differential equations, the compatibility condition
\begin{equation}
A(x, qt)B(x, t)=B(qx, t)A(x, t)
\end{equation}
of \eqref{eq:qLax} reduces to a system of non-linear $q$-difference equations satisfied by entries of $A(x,t)$ and $B(x,t)$.
Actually, the entries of $B(x,t)$ can be expressed by entries of $A(x,t)$.
In Section~\ref{sec:derivation}, we see how to determine the matrix $B(x,t)$ through an example.

\subsection{Spectral types of linear $q$-difference systems}\label{sec:qST}

Here we recall the notion of spectral type of linear $q$-difference systems
introduced in \cite{SY}.
For simplicity, we explain it in the case when $A_0$ and $A_N$ are semisimple.

Consider the following linear $q$-difference system:
\begin{equation}
Y(qx)=A(x)Y(x),
\end{equation}
where $A(x)=A_N x^N+\cdots+A_1x+A_0$ is an $m \times m$ matrix whose entries are polynomials in $x$.
We assume that, for any $a \in \mathbb{C},\, A(a) \ne O$.
Let the eigenvalues of $A_0$ be $\theta_j \ (j=1, \ldots, k)$, and let their multiplicities be $m_j \ (j=1, \ldots, k)$.
Similarly, let the eigenvalues of $A_N$ be $\kappa_j \ (j=1, \ldots, \ell)$, and let their multiplicities be $n_j \ (j=1, \ldots, \ell)$:
\begin{equation}
A_0 \sim
\theta_1 I_{m_1} \oplus \cdots \oplus \theta_k I_{m_k}, \quad
A_N \sim
\kappa_1 I_{n_1} \oplus \cdots \oplus \kappa_\ell I_{n_\ell}.
\end{equation}
Then we define partitions $S_0$ and $S_\infty$ of $m$ as
\begin{equation}
S_0=m_1, \ldots, m_k, \quad
S_\infty=n_1, \ldots, n_\ell.
\end{equation}

Let $Z_A$ be the set of the zeros of $\det A(x)$:
\begin{equation}
Z_A=\{ a \in \mathbb{C} \,|\, \det A(a)=0 \}=
\{ \alpha_1, \ldots, \alpha_p \}.
\end{equation}
We denote by $d_i \, (i=1, \ldots, m)$ the elementary divisors of $A(x)$.
Here we assume that $d_{i+1} | d_i$ (which is opposite to the usual convention).
For any $\alpha_i \in Z_A$, we denote by $\tilde{n}^i_k$ the order of $\alpha_i$ in $d_k$.
Let $\{ n^i_j \}_j$ be the partition conjugate to $\{ \tilde{n}^i_k \}_k$.
Then we define $S_{\mathrm{div}}$ by
\begin{equation}
S_{\mathrm{div}}=n^1_1 \ldots n^1_{k_1}, \ldots, n^p_1 \ldots n^p_{k_p}.
\end{equation}

So far, we have seen the spectral type of both linear differential equations and linear $q$-difference equations.
Concerning the relationship between the two kinds of spectral types, we propose the following conjecture.
\begin{conj}\label{thm:correspondence_of_ST}
Consider an $m \times m$ system of Fuchsian differential equations:
\begin{equation}\label{eq:differential}
\frac{dY}{dx}=\sum_{j=0}^n\frac{\mathcal{A}_j}{x-u_j}Y,  \quad u_0:=0.
\end{equation}
Let $\mathcal{P}_j=m^j_1, \ldots, m^j_{\ell_j}$ ($j=0,\ldots,n,\infty$) be the partition of $m$ corresponding to $\mathcal{A}_j$.
Thus its spectral type reads $\mathcal{P}_0 \, ; \, \ldots \, ; \, \mathcal{P}_n \, ; \, \mathcal{P}_\infty$.

For the equation \eqref{eq:differential}, consider an $m \times m$ linear $q$-difference system
\begin{equation}\label{eq:q-analogue}
Y(qx)=A(x)Y(x), \quad
A(x)=A_nx^n+\cdots+A_0
\end{equation}
with the following spectral type:
\[
S_0=\mathcal{P}_0, \quad
S_\infty=\mathcal{P}_\infty, \quad
S_\mathrm{div}=m^1_1, \ldots, m^1_{\ell_1}, \ldots, m^n_1, \ldots, m^n_{\ell_n}.
\]
Changing the dependent variable as $Z(x)=f(x)Y(x)$ with a suitable scalar function $f(x)$, we have
\begin{equation}\label{eq:q-difference}
\frac{Z(x)-Z(qx)}{(1-q)x}=\frac{1}{(1-q)x}\left\{ I_m-\frac{f(qx)}{f(x)}A(x) \right\}Z(x).
\end{equation}
Then, by taking a continuous limit $q \to 1$ such that
\begin{equation}
\alpha^k_j \ (j=1, \ldots, \ell_k) \to u_k \quad (q \to 1)
\end{equation}
where $\alpha^k_j$ is the zero of $\det A(x)$ corresponding to $m^k_j$,
we can let the equation \eqref{eq:q-difference} tend to \eqref{eq:differential}.
\end{conj}
If the conjecture is true,
it is expected that the connection-preserving deformation equation of \eqref{eq:q-analogue} gives a $q$-analogue of the isomonodromic deformation equation of \eqref{eq:differential}.

\section{Derivation of a $q$-analogue of the matrix $P_\mathrm{VI}$}\label{sec:derivation}
In this section we derive a $q$-analogue of the matrix $P_{\mathrm{VI}}$.
We consider a linear $q$-difference system of spectral type $(m,m; m,m-1,1; m,m,m,m)$,
which by Conjecture~\ref{thm:correspondence_of_ST}
is expected to give a $q$-analogue of the linear system~\eqref{eq:Fuchs_mp} associated with the matrix $P_{\mathrm{VI}}$.

\subsection{Parametrization of linear systems}
First we parametrize linear $q$-difference systems of spectral type $(m,m; m,m-1,1; m,m,m,m)$.
Consider a linear $q$-difference system of the following form:
\begin{align}\label{eq:linear_qmatPVI}
Y(qx)&=A(x)Y(x), \quad
A(x)=A_2x^2+A_1x+A_0, \quad
A_j \in M_{2m}(\mathbb{C}),
\end{align}
where
\begin{align}
A_2=
\begin{pmatrix}
\kappa_1 I_m & O \\
O & K
\end{pmatrix}, \quad
K=
\mathrm{diag}(\overbrace{\kappa_2, \ldots, \kappa_2}^{m-1}, \kappa_3), \quad
A_0=G_0
\begin{pmatrix}
\theta_1 t I_m & O \\
O & \theta_2 t I_m
\end{pmatrix}
G_0^{-1}.
\end{align}
We partition the matrix $A_1$ as follows:
\begin{equation}
A_1=
\begin{pmatrix}
A_{11} & A_{12} \\
A_{21} & A_{22}
\end{pmatrix},
\quad A_{ij} \in M_m(\mathbb{C}).
\end{equation}
Since $S_{\mathrm{div}}=m,m,m,m$, 
the Smith normal form of the polynomial matrix $A(x)$ is the following
(note that we follow the usual convention here):
\begin{equation}
\begin{pmatrix}
I_m & O \\
O & \prod_{i=1}^4 (x-\alpha_i)I_m
\end{pmatrix}.
\end{equation}
We assume that $\alpha_j$'s depend on $t$ as follows:
\begin{equation}
\alpha_j=\left\{
\begin{aligned}
a_j t \quad (j=1,2), \\
a_j \quad (j=3,4).
\end{aligned}
\right.
\end{equation}
We also assume $q\alpha_i \ne \alpha_j \ (i \ne j)$.
The parameter $t$ will play the role of a deformation parameter later,
then the parameters $\theta_j$, $\kappa_j$, and $a_j$'s are independent of $t$.
Since
\begin{equation}
\det A(x)=\kappa_1^m \kappa_2^{m-1}\kappa_3 \prod_{i=1}^4(x-\alpha_i)^m,
\end{equation}
we have
\begin{equation}\label{eq:qFuchsrel}
\kappa_1^m \kappa_2^{m-1}\kappa_3 \prod_{i=1}^4 a_i^m=\theta_1^m\theta_2^m.
\end{equation}

By the assumption above, the matrix $A_0$ can be written as follows:
\begin{equation}
A_0=
\theta_2 t I_{2m}+
\begin{pmatrix}
I_m \\
B_1
\end{pmatrix}
\begin{pmatrix}
(\theta_1-\theta_2)t I_m-C_1B_1 & C_1
\end{pmatrix}.
\end{equation}
Thus the matrix $A(x)$ reads
\begin{equation}
A(x)=
\begin{pmatrix}
\kappa_1 I_m x^2+A_{11}x+(\theta_1 t I_m-C_1B_1) & A_{12}x+C_1 \\
A_{21}x+B_1\{ (\theta_1-\theta_2)t I_m-C_1B_1 \} & Kx^2+A_{22}x+(\theta_2 t I_m+B_1C_1)
\end{pmatrix}.
\end{equation}
In the same way as in~\cite{JS}, we introduce new variables $F, G_1, G_2$, and $W$ by
\begin{equation}
A(x)=
\begin{pmatrix}
WK\{ \kappa_1(xI_m-F)(xI_m-\bm{\alpha})+\kappa_1G_1 \}K^{-1}W^{-1} & WK(xI_m-F) \\
A_{21}x+B_1\{ (\theta_1-\theta_2)t I_m-C_1B_1 \} & K(xI_m-\bm{\beta})(xI_m-F)+KG_2
\end{pmatrix}.
\end{equation}
Below we show that the matrix $A(x)$ can be expressed using $F, G_1, G_2$, and $W$
(indeed, $G_1$ and $G_2$ are written by $F$ and another matrix $G$).
First, we immediately have
\begin{align}
&C_1=-WKF, \quad
A_{12}=WK, \\
&A_{11}=-\kappa_1WK(F+\bm{\alpha})K^{-1}W^{-1}, \quad
A_{22}=-K(F+\bm{\beta}), \\
&\theta_1 t I_m-C_1B_1=\kappa_1WK(F\bm{\alpha}+G_1)K^{-1}W^{-1}, \quad
\theta_2 t I_m+B_1C_1=K(\bm{\beta}F+G_2). \label{eq:alpha_beta}
\end{align}

Let $M_1$ and $M_2$ be the following polynomial matrices:
\begin{align}
M_1&=
\begin{pmatrix}
I_m & O \\
O & G_1K^{-1}
\end{pmatrix}
\begin{pmatrix}
I_m & O \\
Z & I_m
\end{pmatrix}
\begin{pmatrix}
\kappa_1^{-1}WKG_1^{-1}K^{-1}W^{-1} & O \\
O & I_m
\end{pmatrix}, \\
M_2&=
\begin{pmatrix}
I_m & O \\
-\kappa_1(xI_m-\bm{\alpha})K^{-1}W^{-1} & I_m
\end{pmatrix}
\begin{pmatrix}
I_m & -\kappa_1^{-1}WKG_1^{-1}(xI_m-F) \\
O & I_m
\end{pmatrix},
\end{align}
where $Z$ is the $(2,1)$-block of
\begin{equation}
-A(x)
\begin{pmatrix}
I_m & O \\
-\kappa_1(xI_m-\bm{\alpha})K^{-1}W^{-1} & I_m
\end{pmatrix}.
\end{equation}
Then by direct calculation we have
\begin{equation}\label{eq:SmithA}
M_1A(x)M_2=
\begin{pmatrix}
I_m & O \\
O & \tilde{A}(x)
\end{pmatrix}
\end{equation}
with
\begin{equation}
\tilde{A}(x)=(x^3+Q_1x^2+Q_2x+Q_3)(xI_m-F)+G_1G_2.
\end{equation}
Here $Q_1, Q_2$, and $Q_3$ are written as follows:
\begin{align}
Q_1&=-G_1(\bm{\alpha}+\bm{\beta}+F)G_1^{-1}, \\
Q_2&=G_1(G_1+G_2-\kappa_1^{-1}K^{-1}A_{21}A_{12}+F\bm{\alpha}+\bm{\beta}\bm{\alpha}+\bm{\beta}F)G_1^{-1}, \\
Q_3&=-G_1(\bm{\beta}G_1+G_2\bm{\alpha}+\kappa_1^{-1}K^{-1}(A_0)_{21}A_{12}+\bm{\beta}F\bm{\alpha})G_1^{-1},
\end{align}
where $(A_0)_{21}$ is the $(2, 1)$-block of $A_0$.
By the assumption, the Smith normal form of $\tilde{A}(x)$ is $\prod_{j=1}^4(x-\alpha_j)I_m$.
Thus we have
\begin{equation}\label{eq:Smith_relation}
\tilde{A}(x)=\prod_{j=1}^4(x-\alpha_j)I_m
\end{equation}
since $\deg \tilde{A}(x)=4$.
Equating the coefficient matrices of $x^j$ of both sides of \eqref{eq:Smith_relation}, we have
\begin{align}
Q_1&=F+\beta_1, \label{eq:Q1}\\
Q_2&=(F+\beta_1)F+\beta_2, \label{eq:Q2}\\
Q_3&=(F^2+\beta_1 F+\beta_2)F+\beta_3, \label{eq:Q3}\\
G_1G_2&=(F-\alpha_1I)(F-\alpha_2I)(F-\alpha_3I)(F-\alpha_4I) \label{eq:G1G2rel}.
\end{align}
Here the auxiliary parameters $\beta_j$'s are defined by
\begin{equation}
\sum_{j=0}^4 \beta_{4-j}x^j:=\prod_{j=1}^4(x-\alpha_j).
\end{equation}
In particular, because of \eqref{eq:G1G2rel}, we can introduce a new variable $G$ by
\begin{align}
G_1&=q^{-1}\kappa_1^{-1}(F-\alpha_1)(F-\alpha_2)G^{-1}, \\
G_2&=q\kappa_1G(F-\alpha_3)(F-\alpha_4).
\end{align}
By means of the equations \eqref{eq:Q1}, \eqref{eq:Q2}, \eqref{eq:Q3}, and \eqref{eq:alpha_beta}, we have
\begin{align}
B_1&=(\kappa_1^{-1}-K^{-1})^{-1}\{ F^{-1}G_1+G_2F^{-1}-t(\theta_1{\kappa_1}^{-1}+\theta_2K^{-1})F^{-1}-F-G_1^{-1}FG_1-\beta_1 \}K^{-1}W^{-1}, \\
\bm{\alpha}&=(\kappa_1-K)^{-1}\left\{ (\theta_1+\theta_2)t F^{-1}-\kappa_1F^{-1}G_1-KG_2F^{-1}+K(F+G_1^{-1}FG_1+\beta_1) \right\}, \\
\bm{\beta}&=(\kappa_1-K)^{-1}\left\{ -(\theta_1+\theta_2)t F^{-1}+\kappa_1F^{-1}G_1+KG_2F^{-1}-\kappa_1(F+G_1^{-1}FG_1+\beta_1) \right\}, \\
A_{21}&=\kappa_1K\{ G_1+G_2+F\bm{\alpha}+\bm{\beta}F+\bm{\beta}\bm{\alpha}-G_1^{-1}(F^2+\beta_1F+\beta_2)G_1 \}K^{-1}W^{-1},
\end{align}
and
\begin{align}\label{eq:commutation_relation}
F^{-1}GFG^{-1}&=\frac{a_1a_2a_3a_4\kappa_1}{\theta_1\theta_2}K. 
\end{align}
The equation \eqref{eq:commutation_relation} is a commutation relation between $F$ and $G$.

From the above we obtain a parametrization of linear $q$-difference systems with spectral type $(m,m; m,m-1,1; m,m,m,m)$:
\begin{align}
A(x)=A(x,t)&=
\begin{pmatrix}
WK\{ \kappa_1(xI_m-F)(xI_m-\bm{\alpha})+\kappa_1G_1 \}K^{-1}W^{-1} & WK(xI_m-F) \\
\kappa_1(\bm{\gamma}x+\bm{\delta})W^{-1} & K(xI_m-\bm{\beta})(xI_m-F)+KG_2
\end{pmatrix}
\end{align}
with \eqref{eq:commutation_relation}.
Here $\bm{\gamma}$ and $\bm{\delta}$ are given by
\begin{align}
\bm{\gamma}&=K\{ G_1+G_2+F\bm{\alpha}+\bm{\beta}F+\bm{\beta}\bm{\alpha}-G_1^{-1}(F^2+\beta_1F+\beta_2)G_1 \}K^{-1}, \\
\bm{\delta}&=\kappa_1^{-1}\{ t^2\theta_1\theta_2 F^{-1}-\kappa_1 K(G_2+\bm{\beta}F)F^{-1}(G_1+F\bm{\alpha}) \}K^{-1}.
\end{align}

\subsection{The number of accessory parameters}
In the previous subsection,
we have seen that the linear $q$-difference systems with spectral type $(m,m; m,m-1,1; m,m,m,m)$ are expressed by $F$, $G$, and the gauge variable $W$.
In this subsection we see that, by the commutation relation (\ref{eq:commutation_relation}),
$F$ and $G$ can be expressed by $2m$ parameters other than the gauge freedom which comes from $\mathrm{Stab}(K)$.
In other words, the number of accessory parameters of the linear system under consideration is $2m$.

First, using the action of $\mathrm{Stab}(K)$, we can assume that $F$ or $G$ has a particular form.
Here we assume that $F$ has the following form
(in a similar manner to the proof of Proposition~4.2 in \cite{K2}):
\begin{equation}
F=
\begin{pmatrix}
0         & 1        & 0        & \ldots & \ldots & 0 \\
0         & 0        & 1        & \ddots &          & 0 \\
\vdots & \vdots & \ddots & \ddots & \ddots & \vdots \\
0         &  0       &           & \ddots & \ddots & 0 \\
f_2       & f_3      & \ldots &           & f_m      & 1\\
f_1       & f_{m2}  & f_{m3}  & \ldots & f_{m,m-1}& f_{mm}
\end{pmatrix}.
\end{equation}
Note that the principal $(m-1) \times (m-1)$ submatrix of $F$ obtained by deleting the last row and column is in the form of a companion matrix.
Rewrite the commutation relation \eqref{eq:commutation_relation} as
\begin{equation}\label{eq:com_rel_rewrite}
GF-\rho FKG=O,
\end{equation}
where $\rho=a_1a_2a_3a_4\kappa_1(\theta_1\theta_2)^{-1}$.
By the relation \eqref{eq:qFuchsrel}, we have $\det(\rho K)=1$.
Regarding \eqref{eq:com_rel_rewrite} as a system of linear equations with respect to $G=(\bm{g}_1, \ldots, \bm{g}_m)$, we further rewrite it as
\begin{equation}\label{eq:eqFandG}
\Omega_m \bm{g}=\bm{0}, \quad
\Omega_m:= {}^t\! F \otimes I_m-I_m \otimes (\rho FK),
\end{equation}
where
\begin{equation}
\bm{g}=
\begin{pmatrix}
\bm{g}_1 \\
\vdots \\
\bm{g}_m
\end{pmatrix}.
\end{equation}

Now assume that \eqref{eq:eqFandG} has an invertible solution $G$.
Then
\begin{equation}
(G^{\oplus m})^{-1} \Omega_m G^{\oplus m}=
{}^t\!F \otimes I_m-I_m \otimes F.
\end{equation}
By the Cayley-Hamilton theorem, we find that $\mathrm{rank}({}^t\!F \otimes I_m-I_m \otimes F)=m(m-1)$.
Thus we have
\begin{equation}\label{eq:rank_omega}
\mathrm{rank} \, \Omega_m=m(m-1).
\end{equation}
On the other hand, by elementary row operations, $\Omega_m$ can be reduced to
\begin{equation}
\Omega_m \to
\begin{pmatrix}
O & O & \ldots & O & \Phi_F(\tilde{F}) \\
I_m & -\tilde{F} & \ldots & \ldots & * \\
 & I_m & \ddots & & \vdots \\
 & & \ddots & \ddots & \vdots \\
 & & & I_m & f_{mm}-\tilde{F}
\end{pmatrix}.
\end{equation}
We have written $\rho FK$ as $\tilde{F}$ for simplicity.
For an $m \times m$ matrix $M$, we denote the characteristic polynomial of $M$ by $\Phi_M(\lambda)$:
\begin{equation}
\Phi_M(\lambda)=\lambda^m-c_1(M)\lambda^{m-1}+\cdots+(-1)^{m-k}c_{m-k}(M)\lambda^k+\cdots+(-1)^m c_m(M).
\end{equation}
The condition \eqref{eq:rank_omega} implies $\Phi_F(\tilde{F})=O$.
Taking the difference of $\Phi_F(\tilde{F})=O$ and $\Phi_{\tilde{F}}(\tilde{F})=O$, we have
\begin{multline}\label{eq:apF}
-(c_1(F)-c_1(\tilde{F}))\tilde{F}^{m-1}+\cdots+(-1)^{m-k}(c_{m-k}(F)-c_{m-k}(\tilde{F}))\tilde{F}^k+\cdots \\
+(-1)^{m-1}(c_{m-1}(F)-c_{m-1}(\tilde{F}))\tilde{F}=O
\end{multline}
(we have used $c_m(F)=c_m(\tilde{F})$).
Note that
\begin{equation}
\left( \tilde{F}^k \right)_{1 \, m}=
\left\{
\begin{matrix}
0 &  (k=1, \ldots, m-2), \\
\rho^{m-1}\kappa_2^{m-2}\kappa_3 & (k=m-1).
\end{matrix}
\right.
\end{equation}
Hence the $(1,m)$ entry of \eqref{eq:apF} reads
\begin{equation}
c_1(F)-c_1(\tilde{F})=0.
\end{equation}
Multiplying \eqref{eq:apF} by $\tilde{F}$ and again from the $(1,m)$ entry we have
\begin{equation}
c_2(F)-c_2(\tilde{F})=0.
\end{equation}
By repeating this procedure, we finally have
\begin{equation}\label{eq:cond_nontrivial_sol}
c_k(F)-c_k(\tilde{F})=0 \quad (k=1, \ldots, m-1).
\end{equation}
When we regard \eqref{eq:cond_nontrivial_sol} as equations with respect to $f_{m2}, \ldots, f_{mm}$,
they are uniquely expressed by $f_1, \ldots, f_m$.
Conversely, it is obvious that \eqref{eq:cond_nontrivial_sol} implies \eqref{eq:rank_omega}.
Under the condition \eqref{eq:rank_omega},
the equation \eqref{eq:eqFandG} has $m$ independent solutions and thus its general solution contains $m$ arbitrary constants, say $g_1, \ldots, g_m$.
Therefore $F$ and $G$ are expressed by these $2m$ accessory parameters $f_i$'s and $g_i$'s $(i=1, \ldots, m)$.

\subsection{Rationality of $B(x,t)$}\label{sec:determination_of_B}

Let us consider the connection-preserving deformation of \eqref{eq:linear_qmatPVI}.
Let $\sigma_t$ be the $q$-shift operator with respect to $t$, namely, $\sigma_t : t \mapsto qt$.
For a (matrix valued) function $M=M(t)$, we often denote $\sigma_t(M)$ by $\overline{M}$.

By virtue of Proposition~\ref{thm:local_sol}, there exist the following local solutions of \eqref{eq:linear_qmatPVI}
\begin{align}
&Y_0(x,t)=G_0\hat{Y}_0(x,t)x^{D_0}, \quad
D_0=\frac{1}{\log q}
\mathrm{diag}(\overbrace{\log(\theta_1 t),\ldots,\log(\theta_1 t)}^{m}, \overbrace{\log(\theta_2 t),\ldots,\log(\theta_2 t)}^{m}), \label{eq:local_sol_0}\\
&Y_\infty(x,t)=q^{u(u-1)}\hat{Y}_\infty(x,t)x^{D_\infty}, \quad
D_\infty=\frac{1}{\log q}
\mathrm{diag}(\overbrace{\log \kappa_1,\ldots,\log \kappa_1}^{m}, \overbrace{\log \kappa_2,\ldots,\log \kappa_2}^{m-1},\log \kappa_3). \label{eq:local_sol_infty}
\end{align}
We note that
\begin{equation}
\sigma_t(D_0)=D_0+I_{2m}, \quad
\sigma_t(D_\infty)=D_\infty.
\end{equation}
The corresponding connection matrix $P(x,t)$ is defined by $P(x,t)=Y_0(x,t)^{-1}Y_\infty(x,t)$.

Now assume that $\sigma_t(P(x,t))=P(x,t)$.
Then by Proposition~\ref{thm:pseudo-const}, we have
\begin{equation}
(\sigma_t Y_\infty)Y_\infty^{-1}=(\sigma_t Y_0)Y_0^{-1}.
\end{equation}
We denote this matrix by $B(x,t)$:
\begin{equation}\label{eq:defB}
B(x,t):=(\sigma_t Y_\infty)Y_\infty^{-1}=(\sigma_t Y_0)Y_0^{-1}.
\end{equation}
\begin{prop}
The matrix $B(x,t)$ is rational in $x$.
More specifically, it has the following form:
\begin{equation}
B(x,t)=
\frac{x(x^{2m-1}+\tilde{B}_{2m-2}x^{2m-2}+\cdots+\tilde{B}_0)}{(x-qa_1t)^m(x-qa_2t)^m},
\end{equation}
where $\tilde{B}_j$'s are $2m \times 2m$ matrices.
\end{prop}
\begin{proof}
Substituting \eqref{eq:local_sol_0} and \eqref{eq:local_sol_infty} into \eqref{eq:defB},
we see that $B(x,t)$ is equal to
\begin{align}
\hat{Y}_\infty(x,qt) \hat{Y}_\infty(x,t)^{-1}=\overline{G}_0
 \hat{Y}_0(x,qt) x^{\overline{D}_0} x^{-D_0} \hat{Y}_0(x,t)^{-1} G_0^{-1}.
\end{align}
Therefore, as a meromorphic function of $x$ on $\mathbb{P}^1 \setminus \{ 0, \infty \}$,
we find that $B(x,t)$ has the following properties:
\begin{itemize}
\item The poles common to both sides are $x=q a_1 t$, $q a_2 t$, and their orders are at most $m$.
\item (LHS)$=(I+O(x^{-1}))(I+O(x^{-1}))=I+O(x^{-1}) \quad (x \to \infty)$.
\item (RHS)$=x\overline{G}_0(I+O(x))(I+O(x)){G_0}^{-1}=O(x) \quad (x \to 0)$.
\end{itemize}
These conditions imply that $B(x,t)$ is a meromorphic function on $\mathbb{P}^1$, namely, a rational function in $x$ of the following form:
\begin{equation}
B(x,t)
=\frac{x\tilde{B}(x,t)}{(x-qa_1t)^m(x-qa_2t)^m}, \quad
\tilde{B}(x,t)=x^{2m-1}+\tilde{B}_{2m-2}x^{2m-2}+\cdots+\tilde{B}_0.
\end{equation}
\end{proof}

Next we consider the compatibility condition
\begin{equation}\label{eq:CC}
A(x,qt)B(x,t)=B(qx,t)A(x,t).
\end{equation}
This reads
\begin{equation}
q^{2m-1}(x-a_1t)^m(x-a_2t)^mA(x,qt)\tilde{B}(x,t)=(x-qa_1t)^m(x-qa_2t)^m\tilde{B}(qx,t)A(x,t).
\end{equation}
It follows that
\begin{align}
\exists L, \  A(x,qt)\tilde{B}(x,t)&=(x-qa_1t)^m(x-qa_2t)^m(A_2x+L), \label{eq:condition1}\\
\exists L', \ \tilde{B}(qx,t)A(x,t)&=q^{2m-1}(x-a_1t)^m(x-a_2t)^m(A_2x+L'), \label{eq:condition2}\\
L&=L'. \label{eq:condition3}
\end{align}
Conversely, it is easy to see that \eqref{eq:condition1}, \eqref{eq:condition2}, and \eqref{eq:condition3} imply the compatibility condition \eqref{eq:CC}.

The following holds.
\begin{prop}
\begin{align}
(x-qa_1t)^{m-1}(x-qa_2t)^{m-1} \mid \tilde{B}(x,t).
\end{align}
\end{prop}
\begin{proof}
By the compatibility condition, we have
\begin{equation}
A(x,qt)\tilde{B}(x,t)=(x-qa_1t)^m(x-qa_2t)^m(A_2x+L).
\end{equation}
By the relations \eqref{eq:SmithA} and \eqref{eq:Smith_relation}, we have
\begin{align}
\overline{M_1}^{-1}
\begin{pmatrix}
I_m & O \\
O & \prod_{j=1}^4(x-\overline{\alpha_j})I_m
\end{pmatrix}
\overline{M_2}^{-1}
\tilde{B}(x,t)=(x-qa_1t)^m(x-qa_2t)^m(A_2x+L).
\end{align}
Thus we obtain
\begin{align}
\tilde{B}(x,t)&=(x-qa_1t)^{m-1}(x-qa_2t)^{m-1} \nonumber \\
&\quad \times \overline{M_2}
\begin{pmatrix}
(x-qa_1t)(x-qa_2t)I_m & O \\
O & (x-a_3)^{-1}(x-a_4)^{-1}I_m
\end{pmatrix}
 \overline{M_1}(A_2x+L) \nonumber \\
&=(x-qa_1t)^{m-1}(x-qa_2t)^{m-1}
\frac{M(x)}{(x-a_3)(x-a_4)},
\end{align}
where $M(x)$ is some $2m \times 2m$ polynomial matrix in $x$.
Since $(x-qa_1t)^{m-1}(x-qa_2t)^{m-1}$ and $(x-a_3)(x-a_4)$ are relatively prime,
we find $(x-a_3)(x-a_4) \mid M(x)$.
Thus we have
\begin{equation}
\tilde{B}(x,t)=(x-qa_1t)^{m-1}(x-qa_2t)^{m-1}(x+B_0)
\end{equation}
for some matrix $B_0$.
\end{proof}
From the above discussion, $B(x,t)$ is written as 
\begin{equation}\label{eq:Bmat}
B(x,t)=\frac{x(x I+B_0)}{(x-qa_1t)(x-qa_2t)}.
\end{equation}
Here $B_0$ is a $2m \times 2m$ matrix independent of $x$.
We partition it conformably with $A(x,t)$:
\begin{equation}
B_0=
\begin{pmatrix}
B_{11} & B_{12} \\
B_{21} & B_{22}
\end{pmatrix},
\end{equation}
where $B_{ij}$ is $m \times m$.
$B_0$ is determined through the compatibility condition in the next subsection.

\subsection{Derivation of a $q$-analogue of the matrix $P_\mathrm{VI}$}\label{sec:deformation}

Consider the matrix $B(x,t)$ given by \eqref{eq:Bmat}.
In the same manner as in Section~\ref{sec:determination_of_B},
we can show that the compatibility condition $A(x,qt)B(x,t)=B(qx,t)A(x,t)$ is equivalent to the following conditions
\begin{align}
\exists L, \  A(x,qt)(x+B_0)&=(x-qa_1t)(x-qa_2t)(A_2x+L), \label{eq:condition1_new}\\
\exists L', \ (qx+B_0)A(x,t)&=q(x-a_1t)(x-a_2t)(A_2x+L'), \label{eq:condition2_new}\\
L&=L'. \label{eq:condition3_new}
\end{align}
First we assume \eqref{eq:condition1_new}, \eqref{eq:condition2_new}, and \eqref{eq:condition3_new}.
The conditions \eqref{eq:condition1_new} and \eqref{eq:condition2_new} are respectively equivalent to
\begin{align}
A(qa_it,qt)\left( qa_it I_{2m}+B_0 \right)&=O \quad (i=1,2), \label{eq:left_kernel} \\
\left( qa_it I_{2m}+B_0 \right)A(a_it,t)&=O \quad (i=1,2). \label{eq:right_kernel}
\end{align}
The conditions \eqref{eq:left_kernel} and \eqref{eq:right_kernel} imply
\begin{align}
B_0&=-qa_it I_{2m}+
\begin{pmatrix}
I_m \\
\Phi^i
\end{pmatrix}
\Lambda_i
\begin{pmatrix}
\Psi^i & I_m
\end{pmatrix} \quad (i=1,2), \\
\Phi^1&=\kappa_1\{ q^{-1}\kappa_1^{-1}(\overline{F}-qa_2t)\overline{G}^{-1}-qa_1t+\overline{\bm{\alpha}} \}K^{-1}\overline{W}^{-1}, \\
\Phi^2&=\kappa_1\{ q^{-1}\kappa_1^{-1}(\overline{F}-qa_1t)\overline{G}^{-1}-qa_2t+\overline{\bm{\alpha}} \}K^{-1}\overline{W}^{-1}, \\
\Psi^1&=K\{ -a_1t+G_2(F-a_1t)^{-1}+\bm{\beta} \}K^{-1}W^{-1}, \\
\Psi^2&=K\{ -a_2t+G_2(F-a_2t)^{-1}+\bm{\beta} \}K^{-1}W^{-1}.
\end{align}
The column vectors of
\begin{equation}
\begin{pmatrix}
I_m \\
\Phi^i
\end{pmatrix}
\end{equation}
form a basis of the kernel of left multiplication by $A(qa_it, qt)$, 
and the row vectors of
\begin{equation}
\begin{pmatrix}
\Psi^i & I_m
\end{pmatrix}
\end{equation}
form a basis of the kernel of right multiplication by $A(a_it, t)$. 
$\Lambda_1$ and $\Lambda_2$ are $m \times m$ matrices.
Comparing the $(1,2)$-block of
\begin{equation}\label{eq:B0eq}
B_0=
-qa_1t I_{2m}+
\begin{pmatrix}
I_m \\
\Phi^1
\end{pmatrix}
\Lambda_1
\begin{pmatrix}
\Psi^1 & I_m
\end{pmatrix}=
-qa_2t I_{2m}+
\begin{pmatrix}
I_m \\
\Phi^2
\end{pmatrix}
\Lambda_2
\begin{pmatrix}
\Psi^2 & I_m
\end{pmatrix},
\end{equation}
we have $\Lambda_1=\Lambda_2=B_{12}$.

On the other hand, by comparing the coefficients of $x^2$ of both sides of \eqref{eq:condition1_new},
we obtain
\begin{equation}\label{eq:L}
A_2B_0+\overline{A_1}=L-q(a_1+a_2)t A_2.
\end{equation}
Also, from the coefficients of $x^2$ of \eqref{eq:condition2_new} we have
\begin{equation}\label{eq:L'}
qA_1+B_0A_2=qL'-q(a_1+a_2)t A_2.
\end{equation}
Eliminating $L=L'$ from \eqref{eq:L} and \eqref{eq:L'}, we have
\begin{equation}\label{eq:B0eq-2}
(qA_2)B_0-B_0A_2=q(A_1-\overline{A_1})+q(1-q)(a_1+a_2)tA_2.
\end{equation}
From the relation \eqref{eq:B0eq-2} we obtain
\begin{align}
B_{11}&=\frac{q}{q-1}\left\{ \overline{W}K(\overline{\bm{\alpha}}+\overline{F})K^{-1}\overline{W}^{-1}-WK(\bm{\alpha}+F)K^{-1}W^{-1} \right\}
-q(a_1+a_2)t I_m, \label{eq:B11}\\
B_{12}&=q(W-\overline{W})K(q\kappa_1-K)^{-1}. \label{eq:B12}
\end{align}

Comparing the $(2,2)$-block of \eqref{eq:B0eq}, we have
\begin{equation}\label{eq:phieq}
(\Phi^1-\Phi^2)B_{12}=q(a_1-a_2)t.
\end{equation}
Substituting \eqref{eq:B12} into the above, we obtain
\begin{equation}\label{eq:Weq1} 
W^{-1}\overline{W}=q\kappa_1(\overline{G}-K^{-1})^{-1}\left( \overline{G}-\frac{1}{q\kappa_1} \right)K^{-1}.
\end{equation}
Comparing the $(1,1)$-block of \eqref{eq:B0eq}, we have
\begin{equation}\label{eq:psieq}
B_{12}(\Psi^1-\Psi^2)=q(a_1-a_2)t.
\end{equation}
Similarly, we get
\begin{equation}\label{eq:Weq2} 
W^{-1}\overline{W}=
K(q\kappa_1-K)^{-1}\{ -KG_2+q\kappa_1(F-a_1t)(F-a_2t) \}
\{ (F-a_1t)(F-a_2t)-G_2 \}^{-1}K^{-1}(q\kappa_1-K)K^{-1}.
\end{equation}
Equating the right-hand sides of \eqref{eq:Weq1} and \eqref{eq:Weq2}, we obtain
\begin{equation}\label{eq:Geq}
\overline{G}KG=\frac{1}{q\kappa_1}(F-a_1t)(F-a_2t)(F-a_3)^{-1}(F-a_4)^{-1}.
\end{equation}
Comparing the $(2,1)$-block of \eqref{eq:B0eq}, we have
\begin{equation}\label{eq:21rel}
\Phi^1B_{12}\Psi^1=\Phi^2B_{12}\Psi^2.
\end{equation}
On the other hand, using the relations \eqref{eq:phieq} and \eqref{eq:psieq}, we have
\begin{align}
&\Phi^1B_{12}\Psi^1=\Phi^1(q(a_1-a_2)t+B_{12}\Psi^2)=
q(a_1-a_2)t \Phi^1+(q(a_1-a_2)t+\Phi^2B_{12})\Psi^2 \nonumber \\
&=q(a_1-a_2)t(\Phi^1+\Psi^2)+\Phi^2B_{12}\Psi^2.
\end{align}
Thus \eqref{eq:21rel} implies
\begin{equation}
\Phi^1+\Psi^2=O.
\end{equation}
It follows that
\begin{equation}\label{eq:alphabar}
\overline{\bm{\alpha}}=-q^{-1}\kappa_1^{-1}(\overline{F}-qa_2t)\overline{G}^{-1}+qa_1t-\kappa_1^{-1}\Psi^2\overline{W}K.
\end{equation}
By the equation \eqref{eq:B11} and $B_{11}=-qa_2t+B_{12}\Psi^2$, we have
\begin{equation}\label{eq:alphabar2}
\overline{W}K(\overline{\bm{\alpha}}+\overline{F})K^{-1}\overline{W}^{-1}=
WK(\bm{\alpha}+F)K^{-1}W^{-1}+(q-1)a_1t+\frac{q-1}{q}B_{12}\Psi^2.
\end{equation}
Eliminating $\overline{\bm{\alpha}}$ using \eqref{eq:alphabar} and \eqref{eq:alphabar2}, we obtain
\begin{equation}\label{eq:Feq}
\overline{F}KF=
\frac{\theta_1\theta_2}{\kappa_1a_1a_2}\left( \overline{G}-t\frac{a_1a_2}{\theta_1} \right)\left( \overline{G}-t\frac{a_1a_2}{\theta_2} \right)
\left( \overline{G}-\frac{1}{q\kappa_1} \right)^{-1}\left( \overline{G}-\rho \right)^{-1}.
\end{equation}
Then $B_{ij}$'s are determined as follows:
\begin{align}
B_{11}&=
qWK(I-\overline{G}K)^{-1}\overline{G}K \left[ K^{-1}\overline{G}^{-1}\{ F-(a_1+a_2)t \}+\bm{\beta} \right]K^{-1}W^{-1}, \\
B_{12}&=
qWK(I-\overline{G}K)^{-1}\overline{G}, \\
B_{21}&=q\kappa_1\left\{ q^{-1}\kappa_1^{-1}(\overline{F}-qa_2t)\overline{G}^{-1}-qa_1t+\overline{\bm{\alpha}} \right\}
(I-q\kappa_1\overline{G})^{-1} \nonumber \\
&\quad \times \overline{G}K
\left\{ K^{-1}\overline{G}^{-1}( F-a_2t )-a_1t+\bm{\beta} \right\}K^{-1}W^{-1}, \\
B_{22}&=\left[ q^{-1}\kappa_1^{-1} \{ \overline{F}-q(a_1+a_2)t \}\overline{G}^{-1}+\overline{\bm{\alpha}} \right]
q\kappa_1\overline{G}(I-q\kappa_1\overline{G})^{-1}.
\end{align}

Conversely, by direct calculation, we can show that the equations \eqref{eq:Weq1}, \eqref{eq:Geq}, and \eqref{eq:Feq}
are also the sufficient condition to the compatibility condition.
Thus we have the following
\begin{thm}
The compatibility condition $A(x,qt)B(x,t)=B(qx,t)A(x,t)$ is equivalent to
\begin{align}
\overline{F}KF&=
\frac{\theta_1\theta_2}{\kappa_1a_1a_2}\left( \overline{G}-t\frac{a_1a_2}{\theta_1} \right)\left( \overline{G}-t\frac{a_1a_2}{\theta_2} \right)
\left( \overline{G}-\frac{1}{q\kappa_1} \right)^{-1}\left( \overline{G}-\rho \right)^{-1}, \label{eq:Feq_thm}\\
\overline{G}KG&=\frac{1}{q\kappa_1}(F-a_1t)(F-a_2t)(F-a_3)^{-1}(F-a_4)^{-1}, \label{eq:Geq_thm}\\
W^{-1}\overline{W}&=q\kappa_1(\overline{G}-K^{-1})^{-1}\left( \overline{G}-\frac{1}{q\kappa_1} \right)K^{-1}.
\end{align}
\end{thm}
We call the equations \eqref{eq:Feq_thm} and \eqref{eq:Geq_thm} (with \eqref{eq:commutation_relation}) $q$-matrix $P_{\mathrm{VI}}$
since they have a continuous limit $q \to 1$ to the matrix $P_{\mathrm{VI}}$ (see Section~\ref{sec:CL}).
Note that this system can be regarded as a non-abelian analogue of Jimbo-Sakai's $q$-$P_{\mathrm{VI}}$, so we also call it matrix $q$-$P_{\mathrm{VI}}$.

Although this system appears to have eight parameters ($\theta_i$'s, $\kappa_i$'s, and $a_i$'s with
a single relation),
the number of parameters can be reduced to five by rescaling $F$, $G$, and $t$.

\section{Continuous limit}\label{sec:CL}

As expected, the system \eqref{eq:Feq_thm} and \eqref{eq:Geq_thm} can be viewed as a $q$-analogue of the matrix $P_{\mathrm{VI}}$ \eqref{eq:MP6-q} and \eqref{eq:MP6-p}.
That is, taking the limit $q \to 1$, one can obtain \eqref{eq:MP6-q} and \eqref{eq:MP6-p} from \eqref{eq:Feq_thm} and \eqref{eq:Geq_thm}. 
In fact, let us define the parameter $\varepsilon$ by $q=e^{-\varepsilon}$.
We set
\begin{align}
\theta_i=e^{-\varepsilon \sigma_i} \ (i=1,2), \quad a_i=e^{\varepsilon \zeta_i} \ (i=1, \ldots, 4), \quad
\kappa_i=e^{\varepsilon \mu_i} \ (i=1,2,3).
\end{align}
Note that the zeros $\alpha_1, \alpha_2$ of $A(x, t)$ tend to $t$ and $\alpha_3, \alpha_4$ tend to 1 when $\varepsilon$ goes to 0.
Moreover, we introduce new dependent variables $Q$ and $P$,
which are related to $F$ and $G$ by
\begin{align}
F=\tilde{Q}, \quad
G=e^{\varepsilon(\zeta_2+\zeta_4+\sigma_2)}e^{\varepsilon(\tilde{Q}\tilde{P}+I)}(\tilde{Q}-e^{\varepsilon \zeta_1} t I)(\tilde{Q}-e^{\varepsilon \zeta_4} I)^{-1}
\end{align}
and
\begin{align}
\tilde{Q}&=(Q-I)^{-1}(Q-tI), \\
\tilde{P}&=\frac{1}{t-1}(Q-I)\left\{ P(Q-I)+\zeta_1+\zeta_4+\mu_1+\sigma_1-(\zeta_1-\zeta_2)Q^{-1}+(\sigma_1-\sigma_2)(t-1)(Q-tI)^{-1} \right\}.
\end{align}

Then, by taking a limit $\varepsilon \to 0$, we find that $Q$, $P$ satisfy the following equations:
{\small
\begin{align}
&t(t-1)\frac{dQ}{dt} \nonumber \\
&=(Q-t)PQ(Q-1)+Q(Q-1)P(Q-t) \nonumber \\
&\quad+(\sigma_1-\sigma_2+1)Q(Q-1)+(\zeta+2\mu_1+\sigma_1+\sigma_2-1)Q(Q-t)+(\zeta_1-\zeta_2)(Q-1)(Q-t), \\
&t(t-1)\frac{dP}{dt} \nonumber \\
&=-(Q-1)P(Q-t)P-P(Q-t)PQ-PQ(Q-1)P \nonumber \\
&\quad-\left[(\sigma_1-\sigma_2+1)\{P(Q-1)+QP\}+(\zeta+2\mu_1+\sigma_1+\sigma_2-1)\{P(Q-t)+QP\}
+(\zeta_1-\zeta_2)\{P(Q-t)+(Q-1)P\}\right] \nonumber \\
&\quad-(\zeta_1+\zeta_3+\sigma_1+\mu_1)(\zeta_1+\zeta_4+\sigma_1+\mu_1),
\end{align}
}

\vspace{-5mm}
\noindent
where $\zeta:=\zeta_1+\cdots+\zeta_4$.
These equations coincide with \eqref{eq:MP6-q} and \eqref{eq:MP6-p} with the following correspondence of the parameters:
\begin{align}
&\sigma_1-\sigma_2 = \theta^0, \ \zeta_3-\zeta_4 = \theta^1, \ \zeta_1-\zeta_2 = \theta^t, \ 
\mu_i+\zeta_2+\zeta_4+\sigma_2 = \theta^\infty_i \ (i=1,2,3).
\end{align}

Expanding \eqref{eq:commutation_relation} with respect to $\varepsilon$ and
taking the coefficient of $\varepsilon^1$, we have a commutation relation between $P$ and $Q$:
\begin{equation}
[P, Q]=(\zeta_1+\cdots+\zeta_4+\sigma_1+\sigma_2+\mu_1)I_m+M, \quad
M=
\mathrm{diag}(\overbrace{\mu_2, \ldots, \mu_2}^{m-1}, \mu_3).
\end{equation}

The linear system \eqref{eq:linear_qmatPVI} also admit the continuous limit in a similar way.
To see this, we first change the dependent variable $Y$ to $Z(x)=f(x)Y(x)$, where $f(x)$ is a solution of the following $q$-difference equation
\begin{equation}
f(qx)=\frac{f(x)}{(x-1)(x-t)}.
\end{equation}
For example, we can take
\begin{equation}
f(x)=(x; q)_\infty (x/t; q)_\infty \frac{\vartheta_q(x)}{\vartheta_q(x/t)},
\end{equation}
where
\begin{equation}
(a; q)_\infty =\lim_{n \to \infty} \prod_{j=0}^{n-1}(1-aq^j), \quad
\vartheta_q(x)=\prod_{n=0}^\infty (1-q^{n+1})(1+xq^n)(1+x^{-1}q^{n+1}).
\end{equation}
Then we have
\begin{align}
\frac{Z(x)-Z(qx)}{(1-q)x}=\frac{1}{\varepsilon x(1+O(\varepsilon))}\left\{ I_{2m}-\frac{1}{(x-1)(x-t)}A(x) \right\}Z(x).
\end{align}
Put $W=\varepsilon U^{-1}(Q-I)$.
Define matrices $\mathsf{A}_0$, $\mathsf{A}_1$, and $\mathsf{A}_t$ by
\begin{equation}
\lim_{\varepsilon \to 0}\frac{1}{\varepsilon x}\left\{ I_{2m}-\frac{1}{(x-1)(x-t)}A(x) \right\}=
\frac{\mathsf{A}_0}{x}+\frac{\mathsf{A}_1}{x-1}+\frac{\mathsf{A}_t}{x-t}.
\end{equation}
Then the matrices $\mathsf{A}_0$, $\mathsf{A}_1$, $\mathsf{A}_t$ (almost) coincide with \eqref{eq:LDE_mPVI_res}.
More precisely, we have
\begin{align}
\mathsf{A}_0-\sigma_2 I_{2m}=\mathcal{A}_0, \quad \mathsf{A}_1-\zeta_4 I_{2m}=\mathcal{A}_1, \quad \mathsf{A}_t-\zeta_2 I_{2m}=\mathcal{A}_t.
\end{align}
Thus the resulting system of linear differential equations
\begin{align}
\frac{dZ}{dx}=\left( \frac{\mathsf{A}_0}{x}+\frac{\mathsf{A}_1}{x-1}+\frac{\mathsf{A}_t}{x-t} \right)Z
\end{align}
is of spectral type $m,m \, ; \, m,m \, ; \, m,m \, ; \, m,m-1,1$.
In this case, Conjecture~\ref{thm:correspondence_of_ST} holds.

\end{document}